\documentclass{amsart}
\usepackage{graphicx}
\usepackage{amssymb,amsfonts,amsthm}
\usepackage{amsmath}
\usepackage{calc}
\usepackage{graphics}
\newtheorem{thm}{Theorem}[section]
\newtheorem{cor}[thm]{Corollary}

\newtheorem{prop}[thm]{Proposition}
\theoremstyle{definition}

\theoremstyle{remark}
\newtheorem{rem}[thm]{Remark}
\numberwithin{equation}{section}
\newcommand{\Real}{\mathbb R}

\begin{document}

\title[Higher order energy decay rates for damped wave equations]
{Higher order energy decay rates for damped wave equations with variable coefficients}
\author{Petronela Radu, Grozdena Todorova and Borislav Yordanov}
\address{UN-Lincoln, UT-Knoxville, UT-Knoxville}
\email{pradu@math.unl.edu, todorova@math.utk.edu, yordanov@math.utk.edu}
\date{\today}
%\thanks{}%
%\subjclass{}%
%\keywords{}%
%\dedicatory{}%
%\commby{*}%
% ----------------------------------------------------------------
\begin{abstract}
Under appropriate assumptions the energy of wave equations
with damping and variable coefficients $c(x)u_{tt}-\hbox{div}(b(x)\nabla
u)+a(x)u_t =h(x)$ has been shown to decay. Determining the rate of decay for the higher order energies involving the $k$th order spatial and time derivatives has been an open problem with the exception of some sparse results obtained for $k=1,2,3$. We establish estimates that optimally relate the higher order energies with the first order energy by carefully analyzing the effects of linear damping. The results concern weighted (in time) and also pointwise (in time) energy decay estimates. We also obtain $L^\infty$ estimates for the solution $u$. As an application we compute explicit decay rates for all energies which involve the dimension $n$ and the bounds for the coefficients $a(x)$ and $b(x)$ in the case $c (x)=1$ and  $h(x)=0.$
\end{abstract}
\maketitle
% ----------------------------------------------------------------

\section{Introduction}

We will study hyperbolic equations of the form
\begin{equation}\label{waveL}
 c(x)u_{tt}-\hbox{div}(b(x)\nabla u)+a(x)u_t =h(x,t),
 \quad x\in {\bf R}^n,\ \ t>0,
\end{equation}
where $a,b,c,$ and $h$ satisfy the assumptions listed below.
This system is generally accepted as a model of wave propagation in a heterogeneous medium with friction (given by $a(x)u_t$) and source terms $h(x,t)$. The coefficient $c(x)$ accounts for variable mass density, while $b(x)$ is responsible for temperature changes as the wave travels in space (see the derivation in \cite{Ika}). More surprisingly, (\ref{waveL}) has been considered as a model for {\it heat propagation} by Cattaneo \cite{C}, and independently by Vernotte \cite{V}. In this setting the constitutive equation for the flux $q$ as given by
\begin{equation}\label{CV}
q_t+ a(x) q=b(x)\nabla u
\end{equation}
replaces the classical Fourier's law $q= b(x)\nabla u$. The Cattaneo-Vernotte equation
(\ref{CV}) is usually considered as appropriate to describe unsteady heat conduction. For more on the applicability of this model see \cite{D, JP1, JP2}. The finite speed of propagation feature that the solutions enjoy under this formulation renders this system as a more accurate way to study the real effects of heat diffusion. We will show later how these aspects regarding the physical interpretation of (\ref{waveL}) tie in with the mathematical analysis of finding the decay rates, but we believe that this paper offers compelling evidence as to why the diffusion phenomenon is an essential feature for the damped wave equations.

 The literature is replete with results on energy estimates for hyperbolic systems with linear damping, but they mainly concern the case when $b(x)=c(x)=1$ and $h(x,t)=0$.
In \cite{RTY} we considered wave equations with variable coefficients in the case $c(x)=1,$ $h(x,t)=0$ for which we obtained explicit decay rates that depend on the dimension and the growth of the coefficients. The main goal of this work is to show how the energies are related to each other; more precisely, we deduce decay rates for the higher order energies in terms of the decay rate for the first energy $E(t)$. This issue was addressed for example in \cite{Nak} where the author shows that for $c(x)=b(x)=1,$ and $h(x,t)=0$ the energy has the following decay rate
\[
 E(t;u): = \frac{1}{2}(||u_t(t)||^2+||\nabla u(t)||^2) \lesssim (1+t)^{-1},
\]
while for the second order energy the decay rate improves so one has
\[
 E(t;u_t)=\frac{1}{2}(||u_{tt}(t)||^2+||\nabla u_t(t)||^2) \lesssim (1+t)^{-2}.
\]
By using these estimates in the equation one consequently obtains
\[
||\Delta u(t)||^2 \lesssim (1+t)^{-1}.
\]
 The idea of using the diffusion effect in order to show a faster decay for the higher order energies is also present in \cite{FM} where the authors show that for the elasticity system with linear damping $u_t$ one has
 \[
 E(t;u)+E(t;u_t)+E(t;u_{tt}) \lesssim (1+t)^{-2}.
 \]

One expects, however, that the gain in the decay rate when moving from the first energy to the second should be $t^{-2}$. This is motivated by the fact that the long time behavior of the linearly damped wave equation with constant coefficients
\begin{equation}\label{dampwave}
u_{tt}-\Delta u +a u_t=0
\end{equation}
 resembles the behavior of the corresponding parabolic equation
 \begin{equation}\label{heat}
-\Delta u+ au_t =0,
 \end{equation}
  as it was suggested in \cite{MY}, \cite{Nar}, \cite{TY2}. By estimating the higher order energies for the Gaussian as a solution of (\ref{heat}), one notices that $E(t;\partial_t^{k+1} u)$ decays faster than $E(t;\partial_t^k u)$ by a factor of $t^{-2}$. This suggests that a similar phenomena should occur for (\ref{dampwave}). Indeed, our paper shows that this conjecture is correct, so in the absence of other effects caused by inhomogeneities, one basically has
\[
E(t;\partial_t^k u) \lesssim E(t,u) (1+t)^{-2k}, \quad k=0,1,2,...
\]
 In this spirit, by matching the gain in decay $t^{-2}$ of $E(t;\partial_t^{k+1} u)$ versus $E(t;\partial_t^k u)$ obtained for diffusion problems we show the {\it optimality} of our decay results for higher order energies.  One of the most important contributions of this paper is developing a methodology for determining the rates of decay for energies of \textbf{all} orders for a variety of hyperbolic systems with linear damping.

The paper is organized as follows. In the next section we present the assumptions for our results, after which we present the main results regarding the decay of the higher order energies involving time derivatives. The weighted decay of the spatial derivatives is obtained for orders $k=2$ and $k=4$ from which we deduce $L^{\infty}$ estimates of the solution in dimension $n=3$. We apply these results in the case $c=1, h=0$ in section 6. In the Appendix one can find the theorem and proof regarding the finite speed of propagation of solutions of (\ref{waveL}) for variable $b$ and $c$.

\section{Main Assumptions}

Let $a,$ $b,$ and $c$ be smooth coefficients satisfying the conditions:
\begin{eqnarray}\label{abcL}
 \nonumber & & a_0(1+|x|)^{-\alpha}\leq a(x) \leq a_1(1+|x|)^{-\alpha},\\
          & & b_0(1+|x|)^{\beta}\leq  b(x)  \leq b_1(1+|x|)^{\beta},\\
 \nonumber & & c_0(1+|x|)^{-\gamma}\leq  c(x) \leq c_1(1+|x|)^{-\gamma},
\end{eqnarray}
where $a_i, b_i, c_i,$ $i=0,1$ are positive constants and the exponents $\alpha, \beta, \gamma$ satisfy:
\begin{align}\label{abg}
2-\beta-\gamma<2, \quad 2\alpha+\beta-\gamma<2.
\end{align}

We consider initial data
\begin{equation}\label{dataL}
 u|_{t=0}=u_0\in H^1(\Real^n),\quad u_t|_{t=0}= u_1\in
L^2(\Real^n),
\end{equation}
and a source $h\in L^\infty([0,\infty),L^2({\bf R}^n)).$ The three
functions have compact support:
\begin{eqnarray}\label{cs}
& & u_0(x)=0,\ \ u_1(x)=0\ \  \hbox{for}\ \ |x|>R, \\
& & h(x,t)=0\ \ \hbox{for}\ \  q(|x|)>t+q(R),
\end{eqnarray}
where $q$ which measures the increase in the support of $u$ is given in the Appendix.
\\

\section{Weighted $L^2$ estimates for time derivatives}

The strategy behind determining the rate of decay for $E(t;\partial_t^{k} u)$ in terms of the decay for $E(t; u)$ is based on some simple facts. First, since $\partial^{k}_t u$ solves the equation (\ref{dataL}) with the different inhomogeneity $\partial^{k}_t h$, the higher order energy $E(t;\partial_t^{k} u)$ satisfies (almost) the same decay estimate as $E(t;u)$ with possibly different constants (the actual decay will depend on $\partial^{k}_t h$). The question is whether these decay rates can be improved similarly to the case $a=b=c=1$ and $h=0$ where the Fourier representation of $u$ shows faster decay rates of norms involving higher derivatives.

Higher order spatial norms are estimated from
$$\hbox{div}(b(x)\nabla u)=a(x)u_t+c(x)u_{tt}-h(x,t),$$
%\begin{eqnarray*}
% \hbox{div}(b(x)\nabla u) & = &a(x)u_t+c(x)u_{tt}-h(x,t),\\
% \nabla \frac{\hbox{div}(b(x)\nabla u)}{a(x)} & = &\nabla %\left(u_t+\frac{c(x)}{a(x)}u_{tt}-\frac{h(x,t)}{a(x)}\right).
%\end{eqnarray*}
since taking directly $x$-derivatives produces new terms and changes the dissipative nature
of this equation. In other words, we will express the $x$-derivatives in terms of the $t$-derivatives.

Let us recall the definition
\[
E(t;u)=\frac{1}{2}\int(cu_t^2+b|\nabla u|^2)\: dx
\]
and energy identity
\begin{equation}\label{egid}
\frac{d}{dt}E(t;u)+\int au_t^2\: dx =\int hu_t\: dx.
\end{equation}
A simple consequence is the following.

\begin{prop}
\label{propE1} Assume that $0\leq T_0\leq T$ and $\mu\geq 0.$ Then
\begin{eqnarray*}
\int_{T_0}^T  \int (1+t)^{\mu} a(x) u_t^2\: dxdt & \lesssim & (1+T_0)^{\mu}E(T_0;u)+\int_{T_0}^T (1+t)^{\mu-1}E(t;u) dt\\
 & & +\int_{T_0}^{T}\int (1+t)^{\mu} \frac{h^2(x,t)}{a(x)}\: dxdt.
\end{eqnarray*}
\end{prop}

{\bf Proof.} The energy identity (\ref{egid}) implies
\[
\frac{d}{dt}[(1+t)^{\mu} E(t;u)]+(1+t)^{\mu} \int au_t^2 dx =  (1+t)^{\mu} \int hu_t dx+\mu(1+t)^{\mu-1} E(t;u).
\]
Notice that $|hu_t|\leq au_t^2/2+a^{-1}h^2/2.$ The result follows from integration on $[T_0,T].$
\qed

We can now proceed with higher-order norms. Let $v=u_t$ to simplify notations. Then
\[
 c(x)v_{tt}-\hbox{div}(b(x)\nabla v)+a(x)v_t =h_t(x,t),
 \quad x\in {\bf R}^n,\ \ t>0,
\]
and
\begin{equation}
\label{E(v)}
\frac{1}{2}\frac{d}{dt}\int(cv_t^2+b|\nabla v|^2)\: dx+\int av_t^2\: dx =\int h_tv_t\: dx.
\end{equation}
In order to obtain integrals of $E(t;v),$ we multiply the equation
for $v$ with $\frac{1}{2}W(t)v$ where $W$ is to be chosen later. Integration on ${\bf R}^n$
yields
\begin{eqnarray}
\nonumber
& & \frac{1}{2}\frac{d}{dt}\int\left(Wcv_tv+\frac{Wa-W_tc}{2}v^2\right)\: dx-\int Wcv_t^2\: dx\\
\label{H(v)}
& & +\frac{1}{2}\int \left(Wcv_t^2+Wb|\nabla u|^2+\frac{W_{tt}c-W_ta}{2}v^2\right)\: dx\\
& &= \frac{1}{2}\int Wh_tv\: dx.
\nonumber
\end{eqnarray}
We will choose $W$ satisfying
\begin{equation}\label{wass}
W(t)\leq \inf_{x\in\ {\rm supp}\: u(\cdot,t)}\frac{a(x)}{c(x)},\quad W_{tt}c-W_ta\geq 0 .
\end{equation}

The support of the solution $u$ at time $t$ is given by
\[
{\rm supp}\: u(\cdot,t)=\{x\in \Real^n: |x|\leq R+Ct^{\frac{2}{2-\beta-\gamma}}\}
\]
for some $C>0$, as a consequence of the Corollary \ref{q>} from the Appendix.

Adding (\ref{E(v)}) and (\ref{H(v)}), we have
\begin{eqnarray}
\nonumber
& & \frac{1}{2}\frac{d}{dt}\int\left(cv_t^2+b|\nabla v|^2+ Wcv_tv+\frac{Wa-W_tc}{2}v^2\right)\: dx
+\int (a-Wc)v_t^2\: dx\\
\label{EH(v)}
& & +\frac{1}{2}\int \left(Wcv_t^2+Wb|\nabla v|^2+\frac{W_{tt}c-W_ta}{2}v^2\right)\: dx\\
& = & \int \left(h_tv_t+\frac{1}{2}Wh_tv\right)\: dx.
\nonumber
\end{eqnarray}
Assume that $W$ also satisfies
\begin{equation}\label{wass1}
0\leq cv_t^2+b|\nabla v|^2+ Wcv_tv+\frac{Wa-W_tc}{2}v^2
\leq 2 cv_t^2+b|\nabla v|^2+C_0Wav^2,
\end{equation}
where $C_0$ is a constant; later we give explicit conditions in terms of $W,$ $a$ and $c.$

Next multiply identity (\ref{EH(v)}) by $(1+t)^{\nu}$ and integrate on $[T_0,T]$:
\begin{eqnarray*}
\nonumber
& & \frac{1}{2}\int_{T_0}^T(1+t)^{\nu}\frac{d}{dt}\int\left(cv_t^2+b|\nabla v|^2+
Wcv_tv+\frac{Wa-W_t c}{2}v^2\right)\: dxdt \\
\label{EH(v)1}
& & +\frac{1}{2}\int_{T_0}^T \int (1+t)^{\nu}\left(Wcv_t^2+Wb|\nabla v|^2\right)\: dxdt\\
& \leq & \int_{T_0}^T\int (1+t)^{\nu}\left(h_tv_t+\frac{1}{2}Wh_tv\right)\: dxdt.
\nonumber
\end{eqnarray*}
Hence,
\begin{eqnarray*}
& & \frac{1}{2}\int_{T_0}^T \int (1+t)^{\nu}\left(Wcv_t^2+Wb|\nabla v|^2\right)\: dxdt \\
& \leq & \frac{1}{2}(1+T_0)^{\nu}\left[\int \left(cv_t^2+b|\nabla v|^2+ Wcv_tv+\frac{Wa-W_tc}{2}v^2\right)\: dx\right]_{t=T_0} \\
& & +\frac{\nu}{2}\int_{T_0}^T \int (1+t)^{\nu -1}\left( cv_t^2+b|\nabla v|^2+
Wcv_tv+\frac{Wa-W_tc}{2}v^2\right)\: dxdt\\
& &+ \int_{T_0}^T\int (1+t)^{\nu}\left(h_tv_t+\frac{1}{2}Wh_tv\right)\: dxdt.
\end{eqnarray*}
Our assumptions on the quadratic form yield a simpler estimate:
\begin{eqnarray*}
& &\int_{T_0}^T \int (1+t)^{\nu}\left(Wcv_t^2+Wb|\nabla v|^2\right)\: dxdt \\
& \leq & (1+T_0)^{\nu}\left[\int \left(cv_t^2+b|\nabla v|^2+ C_0Wav^2\right)\: dx\right]_{t=T_0} \\
& & +\nu\int_{T_0}^T \int (1+t)^{\nu-1}\left(cv_t^2+b|\nabla v|^2+C_0Wav^2\right)\: dxdt\\
& &+ \int_{T_0}^T\int (1+t)^{\nu}\left(2h_tv_t+Wh_tv\right)\: dxdt.
\end{eqnarray*}

Choose $T_0$ sufficiently large, such that
\begin{equation}\label{nu}
(1+t)^{\nu}W >2\nu(1+t)^{\nu-1}\quad \text{                       for   } t>T_0,
\end{equation}
so by combining similar terms we obtain
\begin{eqnarray}\label{weg}
\nonumber& &\frac{1}{2}\int_{T_0}^T \int (1+t)^{\nu}\left(Wcv_t^2+Wb|\nabla v|^2\right)\: dxdt \\
& &\leq  (1+T_0)^{\nu}\left[\int \left(cv_t^2+b|\nabla v|^2+ C_0Wav^2\right)\: dx\right]_{t=T_0} \\
\nonumber& & +\nu C_0\int_{T_0}^T \int (1+t)^{\nu-1}Wav^2\: dxdt+ \int_{T_0}^T\int (1+t)^{\nu}\left(2h_tv_t+Wh_tv\right)\: dxdt.
\end{eqnarray}

We can estimate the last integral as follows:
\begin{eqnarray}\label{last}
\nonumber&&\int_{T_0}^T\int (1+t)^{\nu}\left(2h_tv_t+Wh_tv\right) dxdt \\
&&\leq \int_{T_0}^T \int \left[ \varepsilon (1+t)^{\nu} c W v_t^2 +C(\varepsilon)(1+t)^{\nu} (cW)^{-1}h_t^2\right] dxdt\\
\nonumber&&+ \int_{T_0}^T\int \left[ \frac{1}{2}(1+t)^{\nu-1}Wav^2 +\frac{1}{2}(1+t)^{\nu+1}Wa^{-1}h_t^2 \right] dx dt
\end{eqnarray}
where $0<\varepsilon<1/4$.

Choose a smooth function $W$ such that it satisfies
\begin{equation}\label{wexp}
w_1(1+t)^{-\omega}\leq W(t)\leq w_2 (1+t)^{-\omega}
\end{equation}
for some $0<\omega<1$ and $w_1,w_2>0$. This implies that
\[
(1+t)^{\nu} (cW)^{-1}h_t^2 \lesssim (1+t)^{\nu+1}Wa^{-1}h_t^2,
\]
which together with (\ref{last}) simplifies (\ref{weg}) to the following:
\begin{eqnarray}\label{weg1}
\nonumber& &\int_{T_0}^T \int (1+t)^{\nu}\left(Wcv_t^2+Wb|\nabla v|^2\right)\: dxdt \\
& &\lesssim   (1+T_0)^{\nu}\left[\int \left(cv_t^2+b|\nabla v|^2+  Wav^2\right)\: dx\right]_{t=T_0} \\
\nonumber& & +\int_{T_0}^T \int (1+t)^{\nu-1}Wav^2\: dxdt+ \int_{T_0}^T\int (1+t)^{\nu+1}Wa^{-1}h_t^2\: dxdt.
\end{eqnarray}

The growth assumptions from (\ref{wexp}) used in this last inequality yield:
\begin{eqnarray}\label{weg2}
\nonumber& &\int_{T_0}^T \int (1+t)^{\nu-\omega}\left(cv_t^2+b|\nabla v|^2\right)\: dxdt \\
& &\lesssim  (1+T_0)^{\nu}\left[\int \left(cv_t^2+b|\nabla v|^2+ Wav^2\right)\: dx\right]_{t=T_0} \\
\nonumber& & +\int_{T_0}^T \int (1+t)^{\nu-\omega-1}av^2\: dxdt+  \int_{T_0}^T\int (1+t)^{\nu-\omega+1}a^{-1}h_t^2\: dxdt.
\end{eqnarray}

By using the estimate from Proposition \ref{propE1} with $\mu=\nu-\omega-1$ and (\ref{wass}) we derive the following estimate:
\begin{eqnarray}\label{weg3}
& &\int_{T_0}^T \int (1+t)^{\nu-\omega}E(t;v)\: dxdt \lesssim (1+T_0)^{\nu}[E(T_0;v)+E(T_0;u)] \\
\nonumber& & +\int_{T_0}^T (1+t)^{\nu-\omega-2}E(t;u) dt+ \int_{T_0}^T\int (1+t)^{\nu-\omega-1}a^{-1}((1+ t)^2h_t^2+h^2)\: dxdt.
\end{eqnarray}

The proof is complete provided we show:
\\[.2in]
\textbf{Existence of the weight function $W$}.
Let
\begin{equation}\label{weightW}
W(t)=w_0(1+t)^{-\omega},
\end{equation}
where :

(i) $w_1\leq w_0 \leq w_2$ with $w_1, w_2$ given by (\ref{wexp}).

(ii) the exponent $\omega$ is chosen such that
\begin{equation}\label{omega}
\max\left\{0, \frac{2(\alpha-\gamma)}{2-\beta-\gamma}\right\}<\omega <1.
\end{equation}

\noindent This choice for the weight $W$ satisfies all the constraints as we show below:
\begin{itemize}
\item The inequality (\ref{wass})$_1$. By (\ref{abcL}) it suffices to show:
\[
w_0(1+t)^{-\omega} \leq \frac{a_0}{c_1} \inf_{|x|\leq R+Ct^{\frac{2}{2-\beta-\gamma}}}(1+|x|)^{\gamma-\alpha}.
\]
For $\alpha \leq \gamma$ this is obvious since $(1+t)^{-\omega} \to 0$ as $t\to \infty$, while in the right hand side we have $(1+|x|)^{\gamma-\alpha}>1.$ For $\alpha >\gamma$ we need to show
\[
w_0(1+t)^{-\omega} \leq \frac{a_0}{c_1} (R+Ct^{\frac{2}{2-\beta-\gamma}})^{\gamma-\alpha},
\]
but for sufficiently large times $t$ this holds since   $\omega > \frac{2(\alpha-\gamma)}{2-\beta-\gamma}$ by (\ref{omega}).

\item (\ref{wass})$_2$ holds since  for $W$ given by (\ref{weightW}) we have $W_{tt} \geq 0, \, W_t<0$ and the coefficients $a$ and $c$ are positive.
\item For the left inequality in (\ref{wass1}) we first complete the square in $v$ and $v_t$, so we are left to show
\[
2Wa-2W_tc-W^2 c\geq 0.
\]
By using (\ref{abcL}) and (\ref{weightW}) we reduce the problem of proving the above inequality to showing:
\begin{multline}\label{in1}
2w_0 a_0 (1+t)^{-\omega}(1+|x|)^{-\alpha} +2w_0\omega c_0(1+t)^{-\omega-1}(1+|x|)^{-\gamma} \\-w_0^2c_1(1+t)^{-2\omega}(1+|x|)^{-\gamma} \geq 0.
\end{multline}
This can be simplified to
\begin{equation}\label{in3}
2w_0 a_0(1+|x|)^{-\alpha+\gamma} +2w_0\omega c_0(1+t)^{-1} -w_0 c_1(1+t)^{-\omega} \geq 0.
\end{equation}
The assumption (\ref{omega}) on $\omega$ gives us that the above inequality is true for large times.

We follow a similar approach for the right inequality in (\ref{wass1}); we complete the square so we have:
\[
\left(v_t-\frac{Wv}{2}\right)^2-\frac{Wa-W_tc}{2c}v^2+C_0\frac{Wa}{c}v^2-\frac{W^2}{4} v^2\geq 0.
\]
We need to show
\begin{equation}\label{in2}
(4C_0-2)Wa+2W_tc-W^2c\geq 0
\end{equation}
The constant $C_0$ can be chosen sufficiently large so that we have $4C_0-2>0.$  By (\ref{abcL}) and (\ref{abg}) in order to prove (\ref{in2}) it is enough to show
\[
(4C_0-2)w_0a_0 (1+|x|)^{-\alpha+\gamma} -c_1w_0^2(1+t)^{-\omega}-2\omega w_0c_1 (1+t)^{-1} \geq 0,
\]
which is done exactly as above when proving (\ref{in3}).
\item The condition (\ref{nu}) is satisfied since $\omega<1$.
\end{itemize}
\vspace*{.2in}

Denote by $\theta =\nu-\omega$. From (\ref{weg3}) with a simple induction argument we obtain

\begin{prop}\label{prop1}
Let $a,b,c,$ and $h$ be smooth coefficients which satisfy (\ref{abcL}) and (\ref{abg}), and let $\theta>0$. Then the solution $u$ of equation (\ref{waveL}) with initial conditions (\ref{dataL}), which satisfy (\ref{cs}), satisfies the following weighted energy estimate:
\begin{eqnarray}\label{wegn}
& &\int_{T_0}^T  (1+t)^{\theta+2k}E(t;\partial^{k}_t u)\: dt \lesssim (1+T_0)^{\nu}[\sum_{i=0}^k E(T_0;\partial^{i}_t u)] \\
\nonumber& & +\int_{T_0}^T (1+t)^{\theta}E(t;u) dt+ \int_{T_0}^T\int  (1+t)^{\theta+1}a^{-1}\sum_{i=0}^k [ (1+t)^{2i}(\partial^{i}_t h)^2]\: dxdt.
\end{eqnarray}
\end{prop}

The above arguments allow us to also obtain pointwise decay rates for the energy as stated in the following:
\begin{prop} Under the assumptions of Proposition \ref{prop1}, we have
\begin{eqnarray}\label{pheg}
E(T;\partial^{k}_t u) \lesssim (1+T)^{-\theta-2k-1}\left[(1+T_0)^{\nu}[\sum_{i=0}^k E(T_0;\partial^{i}_t u)] \right. \\
\nonumber \left.+\int_{T_0}^T (1+t)^{\theta}E(t;u) dt+ \int_{T_0}^T\int  (1+t)^{\theta+1}\sum_{i=0}^k\frac{  (1+t)^{2i}(\partial^{i}_t h(x,t))^2}{a(x)}dxdt\right] .
\end{eqnarray}
\end{prop}
\begin{proof} In (\ref{egid}) written for $\partial^{k}_t u$ we apply Young's inequality with appropriate coefficients such that the damping term disappears. This yields
\begin{equation}
E(T;\partial^{k}_t u) \leq E(t;\partial^{k}_t u) + \int_t^T\int \frac{[\partial^k_s h(x,s)]^2}{a(x)} dx ds , \quad 0 <t<T.
\end{equation}
Hence
\begin{eqnarray}
 &&\int_{T_0}^T (1+t)^{\theta+2k}  E(t;\partial^{k}_t u) dt \geq E(T;\partial^{k}_t u)  \int_{T_0}^T(1+t)^{\theta+2k}   dt \\
 \nonumber &&- \int_{T_0}^T (1+t)^{\theta+2k}\int_t^T\int \frac{[\partial^k_s h(x,s)]^2}{a(x)} dx ds dt.
\end{eqnarray}
Simplifying the above inequality and using (\ref{wegn}) we obtain the desired estimate.
\end{proof}

Let us go back to lower order derivatives and estimate
\begin{equation*}
\int a u_t^2 dx = \int  - c u_{tt} u_t + \text{div} (b\nabla u) u_t -h u_t dx
\end{equation*}
which by the Cauchy inequality gives
\begin{eqnarray}
&& \int a u_t^2 dx \leq  \left(\int c u_{tt}^2 dx \right)^{1/2}\left(\int u_t^2 dx\right)^{1/2} \\
&&+ \left(\int b|\nabla u|^2 dx \right)^{1/2}\left(\int b|\nabla u_t|^2 dx\right)^{1/2}
+ \frac{1}{2} \int \frac{h^2}{a} dx+\frac{1}{2}\int au_t^2 dx.
\end{eqnarray}
We obtain the following
\begin{prop} Under the assumptions of Proposition \ref{prop1}, we have
\begin{eqnarray}
&& \int a u_t^2 dx \lesssim (E(t;u)E(t;u_t))^{1/2} + \int \frac{h^2}{a} dx.
\end{eqnarray}
\end{prop}

The explicit decay can be obtained by using (\ref{pheg}) with $k=0$ and $k=1$;
the decay rate of energy associated with the damping term is approximately $-1$ over
the decay rate of total energy, in agreement with Proposition \ref{propE1}
on the ``average" in time decay rates.

\section{Weighted $L^2$ estimates for spatial derivatives}

Decay estimates of higher order spatial derivatives are important for
studying nonlinear perturbations of equation (\ref{waveL}).
Such problems require more regular solutions and naturally lead to $W^{k,p}$ norms with $k\geq 0$ and
$2< p\leq\infty$. Although the latter norms are expected to decay faster than $H^k$ norms,
this is not easy to verify. The first difficulty is that $L^{p}$ norms are not convenient
to estimate by the multiplier method when $p\neq 2$. Here the standard approach
is to use interpolation or embedding, i.e., the Gagliardo-Nirenberg or Sobolev inequalities.
We thus reduce the question to multiplier estimates with losses of derivatives.
The variable coefficients present an additional difficulty as
$x$-differentiation produces terms that change the dissipative form of (\ref{waveL}).
A simple solution is to express $x$-derivatives in terms of $t$-derivatives from the equation:
\[
Mu=\frac{c}{a}u_{tt}+u_t-\frac{h}{a},
\]
where $Mu=a^{-1}\text{div} (b\nabla u).$
The diffusion phenomenon means $Mu\approx u_t$, so second-order
spatial derivatives are related with the first-order time derivative.
Similarly we can derive an identity for $M^2u$.

Below we give two sufficient conditions on $a,$ $b,$ and $c$ to guarantee
decay estimates of $Mu$ and $M^2u$, respectively. Given $\lambda_1,\lambda_2\in [0,1],$
the coefficients satisfy
\begin{eqnarray}
\label{CC1}& &\ \sup_{x\in\ {\rm supp}\: u(\cdot,t)}\
\left[\frac{c(x)}{a(x)} +\frac{b(x)}{a(x)}\left|\nabla\frac{c(x)}{a(x)}\right|^2
\right]\lesssim (1+t)^{\lambda_1},\\
\label{CC2}& &\ \sup_{x\in\ {\rm supp}\: u(\cdot,t)}\
\left[\frac{1}{a(x)}\hbox{div}\left(b(x)\nabla\frac{c(x)}{a(x)}\right)\right]^2 \lesssim (1+t)^{\lambda_2}.
\end{eqnarray}

\begin{prop}
\label{prop5}
Let $Mu=a^{-1}\text{div} (b\nabla u)$ and assume that $u$ is a solution of (\ref{waveL}).

(i) If (\ref{CC1}) holds, then
\[
a(Mu)^2\lesssim (1+t)^{\lambda_1}cu_{tt}^2+au_t^2+\frac{h^2}{a}.
\]
(ii) If (\ref{CC1}) and (\ref{CC2}) hold, then
\begin{eqnarray*}
a(M^2u)^2 & \lesssim &
(1+t)^{3\lambda_1}cu_{tttt}^2+(1+t)^{\lambda_1}\left(cu_{ttt}^2+b|\nabla u_{tt}|^2\right)\\
& & +(1+t)^{\lambda_2}au_{tt}^2\\
& & +\frac{h_t^2}{a}+(1+t)^{2\lambda_1}\frac{h_{tt}^2}{a}+a\left(M\frac{h}{a}\right)^2.
\end{eqnarray*}
\end{prop}

\begin{proof}
$(i)$ The first claim follows from (\ref{CC1}) and
\[
a(Mu)^2\leq 3a\left(\frac{c^2}{a^2}u_{tt}^2+u_t^2+\frac{h^2}{a^2}\right).
\]

$(ii)$ To verify the second claim, we apply $M(uv)=uMv+vMu+2(b/a)\nabla u\cdot\nabla v.$
We have the chain of identities
\begin{eqnarray*}
M^2u & = & M\left(\frac{c}{a}u_{tt}+u_t-\frac{h}{a}\right)\\
     & = &u_{tt}M\frac{c}{a}+\frac{c}{a}Mu_{tt}+2\frac{b}{a}\nabla\frac{c}{a}\cdot\nabla u_{tt}+Mu_t-M\frac{h}{a}.
\end{eqnarray*}
Using the expression for $Mu,$ we further obtain
\begin{eqnarray*}
M^2u & = & u_{tt}M\frac{c}{a}+\frac{c}{a}\left(
\frac{c}{a}u_{tttt}+u_{ttt}-\frac{h_{tt}}{a}\right) \\
& & + 2\frac{b}{a}\nabla\frac{c}{a}\cdot\nabla u_{tt}+
\left(\frac{c}{a}u_{ttt}+u_{tt}-\frac{h_t}{a}\right)-M\frac{h}{a}\\
& = & \frac{c^2}{a^2}u_{tttt}+2\frac{c}{a}u_{ttt}+\left(M\frac{c}{a}+1\right)u_{tt}
+2\frac{b}{a}\nabla\frac{c}{a}\cdot\nabla u_{tt}\\
& & -\frac{h_t}{a}-\frac{c}{a}\frac{h_{tt}}{a}-M\frac{h}{a}.
\end{eqnarray*}
The square of $M^2u$ is bounded by the sum of squares times $7$:
\begin{eqnarray*}
a(M^2u)^2 & \leq & 7a\left[
\frac{c^4}{a^4}u_{tttt}^2+4\frac{c^2}{a^2}u_{ttt}^2+\left(M\frac{c}{a}+1\right)^2u_{tt}^2
\right]\\
& & +7a\left[4\frac{b^2}{a^2}\left|\nabla\frac{c}{a}\right|^2|\nabla u_{tt}|^2
+\frac{h_t^2}{a^2}+\frac{c^2}{a^2}\frac{h_{tt}^2}{a^2}+\left(M\frac{h}{a}\right)^2\right].
\end{eqnarray*}
Finally, we rewrite last estimate to match conditions (\ref{CC1}) and (\ref{CC2}):
\begin{eqnarray*}
a(M^2u)^2 & \lesssim &
\frac{c^3}{a^3}\cdot cu_{tttt}^2+\frac{c}{a}\cdot cu_{ttt}^2+\left[\left(M\frac{c}{a}\right)^2
+1\right]\cdot au_{tt}^2
\\
& & +\frac{b}{a}\left|\nabla\frac{c}{a}\right|^2\cdot b|\nabla u_{tt}|^2
+\frac{h_t^2}{a}+\frac{c^2}{a^2}\cdot \frac{h_{tt}^2}{a}+a\left(M\frac{h}{a}\right)^2.
\end{eqnarray*}
Thus, we have
\begin{eqnarray*}
a(M^2u)^2 & \lesssim &
(1+t)^{3\lambda_1}cu_{tttt}^2+(1+t)^{\lambda_1}\left(cu_{ttt}^2+b|\nabla u_{tt}|^2\right)
\\
& & +(1+t)^{\lambda_2}au_{tt}^2
+\frac{h_t^2}{a}+(1+t)^{2\lambda_1}\frac{h_{tt}^2}{a}+a\left(M\frac{h}{a}\right)^2.
\end{eqnarray*}

\end{proof}

It is possible to derive expressions for $M^ku,$ $k\geq 3,$ and find conditions on $a,$ $b,$ and $c$
which yield decay estimates of higher order spatial derivatives. However, considering $k\leq 2$ is sufficient for most
applications in $\Real^n$ when $n\leq 3.$
\begin{prop}
Let $a,b,c,$ and $h$ be sufficiently smooth functions which satisfy (\ref{abcL}) and (\ref{abg}).
Define
\[
Mu=\frac{\text{div}(b\nabla u)}{a}.
\]
Then the following weighted estimates hold for the solution $u$ of equation (\ref{waveL}) with initial conditions (\ref{dataL}) which satisfy (\ref{cs}):
\begin{eqnarray*}
(i) &   &  \int_{T_0}^T\int (1+t)^{\theta+1} a(Mu)^2\: dx dt  \\
& \lesssim & (1+T_0)^{\theta+1}[\sum_{i=0}^1 E(T_0;\partial^{i}_t u)] +\int_{T_0}^T (1+t)^{\theta}E(t;u)\: dt\\
&    & + \int_{T_0}^T\int  (1+t)^{\theta+1}\sum_{i=0}^1\frac{ (1+t)^{2i}(\partial^{i}_t h)^2}{a}\: dxdt, \\
(ii)&   & \int_{T_0}^T\int (1+t)^{\theta+3-\lambda_2} a(M^2u)^2\: dx dt \\
 &\lesssim & (1+T_0)^{\theta+3}[\sum_{i=0}^3 E(T_0;\partial^{i}_t u)]
+\int_{T_0}^T (1+t)^{\theta}E(t;u) dt\\
& & + \int_{T_0}^T\int  (1+t)^{\theta+1}\sum_{i=0}^2 \frac{(1+t)^{2i}(\partial^{i}_t h)^2}{a}\: dxdt\\
& & +\int_{T_0}^T\int  (1+t)^{\theta+3-\lambda_2}\left[ \frac{h_t^2}{a}+(1+t)^{2\lambda_1}
\frac{h^2_{tt}}{a}+a\left(M\frac{h}{a}\right)^2\right]\: dxdt.
\end{eqnarray*}

\end{prop}

\begin{proof}
$(i)$ Recall the estimate of $u$ in Proposition~\ref{propE1} where $\mu$ is replaced by $\theta+1$:
\begin{eqnarray*}
\int_{T_0}^T  \int (1+t)^{\theta+1} a(x) u_t^2\: dxdt & \lesssim & (1+T_0)^{\theta+1}E(T_0;u)+\int_{T_0}^T (1+t)^{\theta}E(t;u) dt\\
 & & +\int_{T_0}^{T}\int (1+t)^{\theta+1} \frac{h^2(x,t)}{a(x)}\: dxdt.
\end{eqnarray*}
We also need Proposition~\ref{prop1} with $k=1$:
\begin{eqnarray*}
\int_{T_0}^T  (1+t)^{\theta+2}E(t;u_t)\: dt &\lesssim &(1+T_0)^{\nu}[\sum_{i=0}^1 E(T_0;\partial^{i}_t u)]
+\int_{T_0}^T (1+t)^{\theta}E(t;u) dt\\
& & + \int_{T_0}^T\int  (1+t)^{\theta+1}\sum_{i=0}^1 \frac{(1+t)^{2i}(\partial^{i}_t h)^2}{a}\: dxdt.
\end{eqnarray*}
Since $\lambda_1\leq 1,$ the claim follows from Proposition~\ref{prop5}$(i)$.

$(ii)$ A simple adaptation of Proposition~\ref{propE1} to $u_t$ yields
\begin{eqnarray*}
\int_{T_0}^T  \int (1+t)^{\theta+3} au_{tt}^2\: dxdt & \lesssim &
 (1+T_0)^{\theta+3}E(T_0;u_t)+\int_{T_0}^T (1+t)^{\theta+2}E(t;u_t) dt\\
 & & +\int_{T_0}^{T}\int (1+t)^{\theta+3} \frac{h^2_t}{a}\: dxdt.
\end{eqnarray*}
Applying Proposition~\ref{prop1} to the integral of $E(t;u_t),$ we obtain
\begin{eqnarray*}
\int_{T_0}^T  \int (1+t)^{\theta+3} au_{tt}^2\: dxdt & \lesssim &
 (1+T_0)^{\nu}[\sum_{i=0}^1 E(T_0;\partial^{i}_t u)]
+\int_{T_0}^T (1+t)^{\theta}E(t;u) dt\\
& & + \int_{T_0}^T\int  (1+t)^{\theta+1}\sum_{i=0}^1 \frac{(1+t)^{2i}(\partial^{i}_t h)^2}{a}\: dxdt.
\end{eqnarray*}
It is clear from Proposition~\ref{prop5}$(ii)$ and $\lambda_1,\lambda_2\leq 1$ that $au_{tt}^2$ is the main term
in the upper bound of $M^2u$. We readily estimate the remaining terms by
Proposition~\ref{prop1} with $k=2,3.$
\end{proof}

\begin{rem}
It is straightforward to establish pointwise estimates of $\|a^{1/2}Mu\|_{L^2}$ and $\|a^{1/2}M^2u\|_{L^2}$
for large $t.$ Such results will be presented elsewhere.
\end{rem}

\section{$L^\infty$ estimates in dimension $n=3$}

The most important applications of Propositions 3.4 and 4.2 concern $L^\infty$ decay estimates.
We give an example in $n=3,$ although we can treat all $n\leq 7$ due to the embedding $H^4({\bf R}^n)\subset
L^\infty({\bf R}^n)$ for $4>n/2.$ It is important to mention that
the standard Sobolev estimate
$\|u||_{L^\infty}\lesssim (\|u||_{L^2}+\|\Delta u||_{L^2})$ is
not suitable, since the $L^\infty$ norm is expected to decay faster than the $L^2$ norm.

\begin{prop}
\label{linfty}
Assume that $n=3$ and $a,$ $b$ are $C^1$-functions. If
\begin{eqnarray*}
 \|(\nabla \ln b) u\|_{L^2} & \lesssim &\|b^{1/2}\nabla u\|_{L^2},\\
 \|\nabla(a^{1/2}b^{-1}u)\|_{L^2} & \lesssim &\|b^{1/2}\nabla u\|_{L^2},\\
 \|\Delta u\|_{L^2} & \lesssim &\|a^{-1/2}\text{div}(b\nabla u)\|_{L^2},
\end{eqnarray*}
for every sufficiently regular $u,$ then
\[
\|u\|_{L^\infty}^2 \lesssim \|a^{-1/2}\text{div}(b\nabla u)\|_{L^2}\|b^{1/2}\nabla u\|_{L^2}.
\]
Hence $\|u\|_{L^\infty}$ is bounded in terms of $E(t;u)$ and $E(t;u_t)$ whenever
$u$ is a solution of problem (\ref{waveL}), (\ref{dataL}).
\end{prop}

\begin{proof} We set
\[
a^{-1/2}\text{div}(b\nabla u)=f
\]
and multiply with $u$ to obtain
\[
-\Delta u^2=-2|\nabla u|^2-2\frac{a^{1/2}}{b}fu+2\frac{\nabla b}{b}\cdot u\nabla u.
\]
The well-known formula
\[
u^2(x)=\frac{1}{4\pi} \int \frac{-\Delta u^2(y)}{|x-y|}\: dy,
\]
which is valid for compactly supported $u$, yields the estimate
\begin{eqnarray}
\label{twoterms}
u^2(x) & \lesssim & \int \frac{|a^{1/2}(y)f(y)u(y)|}{b(y)|x-y|}\: dy
                   +\int \frac{|u(y)| |\nabla \ln b(y)\cdot \nabla u(y)|}{|x-y|}\: dy.
\end{eqnarray}
It follows from the Cauchy inequality that
\begin{eqnarray*}
\int \frac{|a^{1/2}(y)f(y)u(y)|}{b(y)|x-y|} dy & \leq & \|f\|_{L^2}
                             \left(\int a(y)b^{-2}(y)|x-y|^{-2}u^2(y) dy\right)^{1/2}\\
                          & \lesssim & \|a^{-1/2}\text{div}(b\nabla u)\|_{L^2}
                             \left(\int |\nabla(a^{1/2}(y)b^{-1}(y)u(y))|^2 dy\right)^{1/2},
\end{eqnarray*}
where the second factor comes from the Hardy inequality
\[
\int |x-y|^{-2}f^2(y)\: dy \leq 4\int |\nabla f(y)|^2\: dy,\quad f\in H^1({\bf R}^3).
\]
The assumptions on $a$ and $b$ give
\[
\int |\nabla(a^{1/2}(y)b^{-1}(y)u(y))|^2\: dy \lesssim \int b(y)|\nabla u(y)|^2\: dy,
\]
so the final estimate becomes
\begin{equation}
\label{est1}
\int \frac{|f(y)u(y)|}{b(y)|x-y|}\: dy  \lesssim  \|a^{-1/2}\text{div}(b\nabla u)\|_{L^2}
                             \|b^{1/2}\nabla u\|_{L^2}.
\end{equation}

The second integral in (\ref{twoterms}) admits similar estimates.
From the Cauchy and Hardy inequalities, we have
\begin{eqnarray}
\nonumber
  \int \frac{|u(y)||\nabla\ln b\cdot \nabla u(y)|}{|x-y|}\: dy & \lesssim &
  \|(\nabla \ln b) u\|_{L^2}\|\Delta u\|_{L^2}\\
\label{est2}
  & \lesssim & \|b^{1/2} \nabla u\|_{L^2} \|a^{-1/2}\text{div}(b\nabla u)\|_{L^2}.
\end{eqnarray}

Adding estimates (\ref{est1}) and (\ref{est2}) completes the proof.
\end{proof}

It is easy to find more explicit conditions on $a$ and $b$ which will
allow us to use Proposition~5.1. This part is about integration by parts and Hardy's inequality.

\begin{cor}
\label{delta}
Assume that $a\in C^1$ and $b\in C^2$ are such that
\begin{eqnarray*}
(i)& & a(x)\lesssim 1,\quad 1\lesssim b(x),\\
(ii)& & |\nabla(a^{1/2}(x)b^{-1}(x))|+|\nabla \ln b(x)| \lesssim (1+|x|)^{-1},
\end{eqnarray*}
and the matrix with entries
$$(iii)\qquad \frac{\delta_{ij}}{2}\Delta b-b_{x_ix_j},\quad i,j=1,\ldots, n$$
is non-negative definite. Then the conditions on $a$ and $b$ in Proposition~\ref{linfty} hold;
hence the $L^\infty$ estimate also holds.

\end{cor}

\begin{proof} Conditions $(i),$ $(ii),$ and Hardy's inequality readily show that
\[
\|(\nabla \ln b) u\|_{L^2} \lesssim \|b^{1/2}\nabla u\|_{L^2},\quad
 \|\nabla(a^{1/2}b^{-1}u)\|_{L^2} \lesssim \|b^{1/2}\nabla u\|_{L^2}
 \]
for sufficiently regular $u.$ Thus, it remains to verify the condition about $\|\Delta u\|_{L^2}.$

Let $a^{-1/2}\text{div}(b\nabla u)=f$ and multiply this equation with $\Delta u.$ Then
\[
\int b(\Delta u)^2\: dx =\int a^{1/2}f\Delta u\: dx-\int (\nabla b\cdot \nabla u)\Delta u\: dx.
\]
Integrating by parts in the second term on the right side, we obtain
\[
\int (\nabla b\cdot \nabla u)\Delta u\: dx =\frac{1}{2}\int \Delta b|\nabla u|^2\: dx-
\sum_{i,j=1}^n \int b_{x_ix_j}u_{x_i}u_{x_j}\: dx.
\]
Thus,
\begin{eqnarray*}
\int b(\Delta u)^2\: dx  &= & \int a^{1/2}f\Delta u\: dx - \frac{1}{2}\int \Delta b|\nabla u|^2\: dx\\
& & +\sum_{i,j=1}^n \int b_{x_ix_j}u_{x_i}u_{x_j}\: dx.
\end{eqnarray*}
We have from assumption $(iii)$ that
\[
\int b(\Delta u)^2\: dx \leq  \int a^{1/2}f\Delta u\: dx.
\]
Now assumption $(i)$ and Cauchy's inequality imply
\[
\int (\Delta u)^2\: dx \lesssim \int b(\Delta u)^2\: dx
\lesssim \|f\|_{L^2}\|\Delta u\|_{L^2},
\]
so $\|\Delta u\|_{L^2}\lesssim \|a^{-1/2}\text{div}(b\nabla u)\|_{L^2}.$
\end{proof}

\section{Application in the homogeneous case $c=1, h=0$}

In this section we will derive explicit decay estimates for (\ref{waveL})
when $c(x)=1$ and $h(x,t)=0$, i.e. for
\begin{equation}\label{waveL1}
 u_{tt}-\hbox{div}(b(x)\nabla u)+a(x)u_t =0,
 \quad x\in {\bf R}^n,\ \ t>0,
\end{equation}
as an application of the results proven here and in \cite{RTY}.
To state these results we introduce the following:

\textbf{Hypothesis A.} Let $a$ and $b$ satisfy the growth conditions listed in (\ref{abcL}). Then there exists a subsolution $A(x)$ which satisfies
\begin{equation}\label{E}
\text{div}(b(x) \nabla A(x))\geq a(x),\quad x\in {\bf R}^n,
\end{equation}
and has the following properties:
\begin{eqnarray}\label{E:div}
 (a1) & & A(x)\geq 0 \ \ \hbox{for all} \ \ x,\\
 (a2) & & A(x)=O(|x|^{2-\alpha-\beta})\ \ \hbox{for large} \ \ |x|,\\
 (a3) & & \mu: =\liminf_{x\rightarrow\infty}\frac{a(x)A(x)}{b(x)|\nabla A(x)|^2}\ >\ 0.
\end{eqnarray}

As we will see below in Theorem \ref{mainthm}, the quantity $\mu$ gives the decay of the weighted $L^2$ norm of $u$, while $\mu + 1$ gives the decay of the energy of $u$. In several cases one can construct explicit subsolutions $A$ which satisfy (a1)-(a3), so in these cases one can compute the value of $\mu$ as given by the formula
\[
\mu=\frac{2-\alpha}{2-\alpha-\beta}.
\]
These cases are:
\begin{enumerate}
\item $a,b$ are radial functions.
\item $a,b$ are ``separable" with the same angular component, i.e. for $ r=|x|$ and $ \omega=\frac{x}{r}\in S^{n-1}$, where $S$ is the unit sphere in $\Real^n$, there exist functions $a_1,$ $ b_1,$ and $\zeta$ such that
\begin{align*}
a(x)&=a(r,\omega)=a_1(r) \zeta(\omega),\\
 b(x)&=b(r,\omega)=b_1(r) \zeta(\omega).
\end{align*}
\item $b$ is radial and $a$ is arbitrary.
\item $a,b$ are arbitrary (no radial symmetry), $\alpha+\beta <2$ and
\[
\nabla b(x) \cdot \frac{x}{|x|} \geq b_{0}\beta(1+|x|)^{\beta-1}.
\]
\end{enumerate}
For a full discussion of these special situations see Section 7 in \cite{RTY}.

 \begin{thm}\label{mainthm}
Assume that $a$ and $b$ satisfy (\ref{abcL}) with $\alpha,$ $ \beta$ such that
\begin{equation}\label{c1}
\alpha < 1, \quad 0\leq \beta <2, \quad 2\alpha+\beta\leq 2.
\end{equation}
Also, assume that Hypothesis A holds. Then for every $\delta>0$ the solution of (\ref{waveL1}) satisfies
\begin{align*}
 &\int e^{(\mu-\delta)\frac{A(x)}{t}}a(x)u^2\ dx  \leq  C_\delta(\|\nabla
 u_0\|_{L^2}^2+\|u_1\|_{L^2}^2)t^{\delta-\mu},\\
 &\int e^{(\mu-\delta)\frac{A(x)}{t}}(u_t^2+b(x)|\nabla u|^2)\ dx \leq
 C_\delta(\|\nabla u_0\|_{L^2}^2+\|u_1\|_{L^2}^2) t^{\delta-\mu-1}
\end{align*}
for all $t\geq 1.$ The constant $C_{\delta}$ depends also on $R,$ $a,$ $b,$ and $n.$
\end{thm}

An immediate consequence of the above theorem which follows by (a1) is
\begin{cor}\label{cor4} Under the assumptions of Theorem \ref{mainthm}, we have
\begin{align}
 \label{engdec}&\int a(x)u^2\ dx  \leq  C_\delta(\|\nabla
 u_0\|_{L^2}^2+\|u_1\|_{L^2}^2)t^{\delta-\mu},\\
 \label{dampdec}&E(t;u)=\int u_t^2+b(x)|\nabla u|^2 \ dx \leq
 C_\delta(\|\nabla u_0\|_{L^2}^2+\|u_1\|_{L^2}^2) t^{\delta-\mu-1}.
\end{align}
\end{cor}

These results in conjunction with the decay estimates presented in Sections 3--5 yield the following:
\begin{thm} Assume that $a,b$ satisfy (\ref{abcL}) and (\ref{c1}). Then for every $\delta>0$ and $T_0>0$ sufficiently large the following inequalities hold:
\begin{enumerate}
\item[(i)] Weighted (in time) L$^2$ energy decay:
\begin{equation}\label{wegna}
\int_{T_0}^T  (1+t)^{\theta+2k}E(t;\partial^{k}_t u)\: dt \lesssim (1+T_0)^{\nu}\sum_{i=0}^k E(T_0;\partial^{i}_t u)  +\frac{T^{\theta+\delta-\mu}-T_0^{\theta+\delta-\mu}}{\theta+\delta-\mu}.
\end{equation}
\item[(ii)] Pointwise (in time) energy decay:
\begin{equation}\label{phega}
E(T;\partial^{k}_t u) \lesssim (1+T)^{-\theta-2k-1}\left[(1+T_0)^{\nu}\sum_{i=0}^k E(T_0;\partial^{i}_t u) +\frac{T^{\theta+\delta-\mu}-T_0^{\theta+\delta-\mu}}{\theta+\delta-\mu} \right].
\end{equation}
\item[(iii)] Weighted kinetic energy decay (the energy associated with the damping):
\begin{align}\label{damphdec}
\int a u_t^2  dx & \lesssim [E(t;u)E(t;u_t)]^{1/2} \\
\nonumber  &\lesssim (1+T)^{-\theta-2}\left[(1+T_0)^{\nu}\sum_{i=0}^k E(T_0;\partial^{i}_t u) +\frac{T^{\theta+\delta-\mu}-T_0^{\theta+\delta-\mu}}{\theta+\delta-\mu} \right].
\end{align}
\item[(iv)] Decay of spatial derivatives (recall that $Mu=a^{-1}\left[\text{div}(b\nabla u)\right] $):
\begin{eqnarray}
&   &  \int_{T_0}^T\int (1+t)^{\theta+1} a(Mu)^2\: dx dt  \\
\nonumber & \lesssim & (1+T_0)^{\theta+1}\left[\sum_{i=0}^1 E(T_0;\partial^{i}_t u)\right] +
\frac{T^{\theta+\delta-\mu}-T_0^{\theta+\delta-\mu}}{\theta+\delta-\mu},\\
&   & \int_{T_0}^T\int (1+t)^{\theta+3-\lambda_2} a(M^2u)^2\: dx dt \\
\nonumber &\lesssim & (1+T_0)^{\theta+3}\left[\sum_{i=0}^3 E(T_0;\partial^{i}_t u)\right]
+\frac{T^{\theta+\delta-\mu}-T_0^{\theta+\delta-\mu}}{\theta+\delta-\mu}.
\end{eqnarray}

\end{enumerate}
\end{thm}

\begin{rem}
Note from the above decay estimate (\ref{phega}) that the $k$-th order energy has a polynomial decay 
of order $\mu+1+2k-\delta$, exactly $2k$ units below the decay of the first order energy as given by (\ref{engdec}).
The damping term has the decay rate $\mu+2-\delta$, as it follows from (\ref{damphdec}), which is $2$ units below the decay of the $L^2$ norm of $u$ given by (\ref{dampdec}).
Finally, the average decay rates of second and fourth order spatial derivatives resemble 
those of $u_t$ and $u_{tt},$ respectively. Here $(1+t)^{\lambda_2}$ accounts for possibly 
fast oscillations in $b(x)/a(x)$ as $|x|\rightarrow\infty$.       
\end{rem}

\appendix
\section{Bounds on the support of solutions}

Recall the non-homogeneous hyperbolic equation with damping
\begin{equation}
\label{waveLH}
 c(x)u_{tt}-\hbox{div}(b(x)\nabla u)+a(x)u_t =h,
 \quad x\in {\bf R}^n,\ \ t>0,
\end{equation}
where the coefficients $a,$ $b,$ and $c$ satisfy conditions (\ref{abcL}) in the introduction.
The propagation speed is determined by the ratio $c(x)/b(x).$ We assume the following:
\begin{equation}
\label{q}
\hbox{There exists}\ \ q\in C^1,\ \  0<q'(|x|)\leq \sqrt{\frac{c(x)}{b(x)}},\ \ x\in {\bf R}^n.
\end{equation}
A better bound on the propagation speed is given in terms of $q$
satisfying
\[
|\nabla q(x)|=\sqrt{\frac{c(x)}{b(x)}},
\]
see Evans~\cite{E}, but this equation is globally solvable only in special cases.
Fortunately the corresponding inequality is sufficient for propagation speed estimates
and solvable in most cases. It is also convenient to have
bounds depending on $|x|,$ so we look for radial
solutions $q$.

We consider $h\in C((0,\infty),L^2({\bf R}^n))$ and
\begin{equation}
\label{supph}
h(x,t)=0\ \ \hbox{if}\ \ q(|x|)> t+q(R).
\end{equation}

There exist a unique solution $u\in C((0,\infty),H^1({\bf R}^n))\cap C^1((0,\infty),L^2({\bf R}^n))$
whose support is described below.

\begin{prop}
\label{suppL} Let $u$ be a solution of (\ref{waveLH}) with data
$(u,u_t)|_{t=0}=(u_0,u_1),$ such that $(u_0,u_1)\in H^1({\bf R}^n)\times L^2({\bf R}^n)$.
If (\ref{q}) holds, then $u_0(x)=u_1(x)=0$ for $|x|>R$ implies $u(x,t)=0$ for
$q(|x|)>t+q(R).$
\end{prop}

Applying the above result, we can find more explicit bounds when $c(x)/b(x)$ is
not very small.

\begin{cor}
\label{q>}
Let $u$ be a solution of (\ref{waveLH}) with data
$(u,u_t)|_{t=0}=(u_0,u_1),$ such that $(u_0,u_1)\in H^1({\bf R}^n)\times L^2({\bf R}^n)$.
Assume that $b$ and $c$ satisfy (\ref{abcL}) with $\beta+\gamma<2.$ There
exists $q$ satisfying (\ref{q}) and
\[
q(|x|)\geq q_0 |x|^{1-(\beta+\gamma)/2}, \quad x\in {\bf R}^n,
\]
with $q_0>0.$ Hence, $u_0(x)=u_1(x)=0$ for $|x|>R$ implies $u(x,t)=0$ for
$$|x|>\left(\frac{t+q(R)}{q_0}\right)^{2/(2-\beta-\gamma)}.$$
\end{cor}

{\em Proof of Proposition~\ref{suppL}.}
We can assume that $(u_0,u_1)$ and $h$ are more regular, so
$u\in C^1((0,\infty),H^1({\bf R}^n))\cap C^2((0,\infty),L^2({\bf R}^n));$
the general case follows from a standard approximation argument.
Multiplying (\ref{waveLH}) with $u_t$, we obtain
\[
\frac{1}{2}\left(c(x)u_{t}^2+b(x)|\nabla u|^2\right)_t-
\hbox{div}(b(x)u_t\nabla u)+a(x)u_t^2 =0.
\]
This identity will be integrated over the exterior of a suitable
truncated cone. Choose $q$ satisfying (\ref{q}) and consider the conic surface
\[
M(T)=\{(x,t)\: : \: T>t>0\ \hbox{and}\ q(|x|)= t+q(R)\},
\]
which is the lateral boundary of the space-time region
\[
C(T)=\{(x,t)\: : \: T>t>0\ \hbox{and}\ q(|x|)> t+q(R)\}.
\]
The outer normal to $M(T)$ at $(x,t)$ is given by the $(n+1)$-vector
\[
\frac{1}{\sqrt{1+[q'(|x|)]^2}}\left(-q'(|x|)\frac{x}{|x|}, 1\right).
\]

We integrate the energy identity over $C(T)$ and apply the
divergence theorem:
\begin{eqnarray*}
& & \left.\frac{1}{2}\int_{q(|x|)>
t+q(R)}\left(c(x)u_{t}^2+b(x)|\nabla u|^2\right)\: dx\right|_{t=T}
+ \int_{C(T)}a(x)u_{t}^2\: dxdt\\
 & & + \int_{M(T)}\left(\frac{c(x)}{2}u_{t}^2+\frac{b(x)}{2}|\nabla
u|^2+b(x)q'(|x|)u_t\frac{x}{|x|}\cdot \nabla u\right)\:
\frac{dS}{\sqrt{1+[q'(|x|)]^2}}\\
\\
= & & \left.\frac{1}{2}\int_{q(|x|)>
t+q(R)}\left(c(x)u_{t}^2+b(x)|\nabla u|^2\right)\: dx\right|_{t=0}
+\int_{C(T)}hu_{t}\: dxdt,
\end{eqnarray*}
for all $T>0.$ The integral over $M(T)$ is non-negative, since
\[
 \frac{c(x)}{2}u_{t}^2+\frac{b(x)}{2}|\nabla u|^2+b(x)q'(|x|)u_t\frac{x}{|x|}\cdot \nabla
 u\geq 0
\]
for all $u_t$ and $\nabla u$ whenever $(bq')^2\leq cb,$ or $0<q'\leq \sqrt{c/b}.$
Thus,
\begin{eqnarray*}
& & \frac{1}{2}\int_{q(|x|)> T+q(R)}\left(c(x)u_{T}^2+b(x)|\nabla u|^2\right)\: dx
+ \int_{C(T)}a(x)u_{t}^2\: dxdt\\
 &  \leq & \frac{1}{2}\int_{q(|x|)>q(R)}\left(c(x)u_{1}^2+b(x)|\nabla u_0|^2\right)\: dx
 +\int_{C(T)}hu_{t}\: dxdt.
 \end{eqnarray*}

Notice that $|hu_t|\leq au_t^2/2+h^2/(2a).$ Hence,
\begin{eqnarray*}
& & \frac{1}{2}\int_{q(|x|)> T+q(R)}\left(c(x)u_{T}^2+b(x)|\nabla u|^2\right)\: dx\\
&  \leq & \frac{1}{2}\int_{q(|x|)>q(R)}\left(c(x)u_{1}^2+b(x)|\nabla u_0|^2\right)\: dx
 +\int_{C(T)}\frac{h^2}{2a} \: dxdt.
 \end{eqnarray*}
The integrands on the right side are $0$ due to our assumptions on $u_0,$
$u_1,$ and $h.$ This shows that $u(x,T)=0$ if $q(|x|)>T+q(R).$ \qed

{\em Proof of Corollary~\ref{q>}.} We already know that $u(x,t)=0$ if $q(|x|)>t+q(R),$
where $q'(|x|)\leq \sqrt{c(x)/b(x)}.$ It is sufficient to find $q$, such that
\[
q'(|x|) = \sqrt{\frac{c_0}{b_1}}(1+|x|)^{-(\beta+\gamma)/2},\quad q(0)=0,
\]
where the constants are given in (\ref{abcL}). Solving for $q$, we obtain
\[
q(|x|) = \sqrt{\frac{c_0}{b_1}}\frac{2}{2-\beta-\gamma}(1+|x|)^{1-(\beta+\gamma)/2}
\geq q_0|x|^{1-(\beta+\gamma)/2},
\]
for some $q_0>0.$
\qed

% ----------------------------------------------------------------

\end{document}